\documentclass{amsart}

\usepackage{amsmath,amsthm,amsfonts,amssymb,graphicx,psfrag,a4wide}
\usepackage{dsfont}
\usepackage{color}
\usepackage{mathtools}
\usepackage{amscd}
\usepackage{mathrsfs}
\usepackage{enumitem}
\usepackage{mathabx}

\usepackage{hyperref}
\usepackage[alphabetic,initials]{amsrefs}

\psfrag{S1}{$S^1$}

\theoremstyle{plain}
\newtheorem{prop}{Proposition}

\newtheorem{thm}[prop]{Theorem}

\newtheorem*{conju}{Conjecture}
\newtheorem{cor}[prop]{Corollary}

\newtheorem{lemma}[prop]{Lemma}

\theoremstyle{definition}
\newtheorem{example}[prop]{Example}

\theoremstyle{remark}
\newtheorem{remark}[prop]{Remark}

\newcommand{\Diff}{\operatorname{Diff}}

\newcommand{\paper}{_P}

\newcommand{\spine}{_\Sigma}

\newcommand{\SO}{\operatorname{SO}}

\newcommand{\CC}{{\mathbb C}}
\newcommand{\DD}{{\mathbb D}}

\newcommand{\NN}{{\mathbb N}}
\newcommand{\QQ}{{\mathbb Q}}

\newcommand{\ZZ}{{\mathbb Z}}

\newcommand{\uU}{{\mathcal U}}

\newcommand{\p}{\partial}

\newcommand{\defin}[1]{\textbf{#1}}

\newcommand{\CP}{\mathbb{CP}}

\definecolor{Chris}{rgb}{0.5,0,1}
\definecolor{Sam}{rgb}{0,1,0}

\numberwithin{equation}{section}

\title{Spine removal surgery and the geography of symplectic fillings}
\author{Samuel Lisi}
\address{University of Mississippi\\
Department of Mathematics\\
USA }
\email{stlisi@olemiss.edu}
\author{Chris Wendl}
\address{Institut f\"ur Mathematik \\
Humboldt-Universit\"at zu Berlin \\
Germany}
\email{wendl@math.hu-berlin.de}

\thanks{S.L.~was partially supported during this project by the ERC Starting Grant of Fr\'ed\'eric Bourgeois
StG-239781-ContactMath, Vincent Colin's ERC Grant geodycon, and by a University
of Mississippi CLA SRG. 
C.W.~was partially supported by a Royal Society University
Research Fellowship, EPSRC grant EP/K011588/1, and the ERC Grant
TRANSHOLOMORPHIC}

\begin{document}

\maketitle

\begin{abstract}
We prove that for any contact $3$-manifold supported by a spinal open book
decomposition with planar pages, there is a universal bound on the Euler
characteristic and signature of its minimal symplectic fillings.  The proof
is an application of the spine removal surgery operation recently
introduced in joint work of the authors with Van Horn-Morris
\cite{LisiVanhornWendl1}.
\end{abstract}


\section{Introduction}

It was conjectured in 2002 by Stipsicz \cite{Stipsicz:gaugeStein} that
every closed contact $3$-manifold admits a universal bound on the signatures
and Euler characteristics of its possible Stein fillings.  Counterexamples
to this conjecture were found a few years ago by Baykur and Van Horn-Morris
\cite{BaykurVanhorn:large}, but it was also shown by Kaloti \cite{Kaloti:Stein}
that the conjecture holds for the special class of \emph{planar} contact
manifolds, i.e.~those that are supported (in the sense of Giroux \cite{Giroux:ICM})
by a planar open book decomposition.  Kaloti's proof was based on relations
in the mapping class groups of compact planar surfaces with boundary,
using a theorem of the second author \cite{Wendl:fillable} that describes
fillings of planar contact manifolds in terms of Lefschetz fibrations,
together with methods of Plamenevskaya and Van Horn-Morris \cite{PlamenevskayaVanHorn}
to achieve bounds on numbers of positive factorizations.
In this note, we shall use completely different methods to prove the
following generalization of Kaloti's result:

\begin{thm}
\label{thm:finCobordism}
Suppose $(M,\xi)$ is a closed contact $3$-manifold with 
a supporting spinal open book whose pages are planar.
Then there exists a finite list of $4$-manifolds
$W_1',\ldots,W_n'$ with boundary, and a compact $4$-manifold $X$ that
has $-M$ as a boundary component, such that if
$(W,\omega)$ is any minimal strong filling of $(M,\xi)$,
then $W \cup_M X \cong W_j'$ for some $j \in \{1,\ldots,n\}$.
\end{thm}

\begin{cor}
Under the hypotheses of Theorem~\ref{thm:finCobordism},
there exists a number $N \in \NN$ such that for all
minimal strong fillings $(W,\omega)$ of $(M,\xi)$,
$$
|\chi(W)| \le N, \quad\text{ and }\quad |\sigma(W)| \le N.
$$
\end{cor}
\begin{proof}
If $W \cup_M X = W_j'$, then $\chi(W) = \chi(W_j') - \chi(X)$
since $M$ has vanishing Euler characteristic, so this gives the
bound on~$\chi(W)$.  A bound on $\sigma(W)$ follows immediately
from Novikov additivity; in more elementary terms, 
one can also see it from the
observation that the inclusion $W \hookrightarrow W_j'$ induces
an injection of any positive/negative-definite subspace of
$H_2(W;\QQ)$ into $H_2(W_j';\QQ)$, implying
$b_2^\pm(W) \le b_2^\pm(W_j')$.
\end{proof}

Spinal open book decompositions were recently introduced in joint work of
the authors with Jeremy Van Horn-Morris \cite{LisiVanhornWendl1}. 
The main motivation behind them is that they are  
the natural structure one obtains on the
boundary of any Lefschetz fibration whose fibers and base are both compact
oriented surfaces with nonempty boundary.  As with ordinary open books,
a spinal open book on $(M,\xi)$ with pages of genus zero gives rise to a 
well-behaved family of pseudoholomorphic curves in the symplectization
of $(M,\xi)$, and the followup paper \cite{LisiVanhornWendl2} carries out
the program of using this technology to classify fillings of~$(M,\xi)$.
The present paper, however, does not rely on those results:
our proof of Theorem~\ref{thm:finCobordism} is comparatively low-tech.
Its main ingredients are the theory of closed $J$-holomorphic spheres
in the spirit of Gromov/McDuff (cf.~\cite{McDuff:rationalRuled}), and
a surgical operation on spinal open books that was introduced in
\cite{LisiVanhornWendl1}, known as \emph{spine removal surgery}.  This 
operation is a generalization of several previous symplectic cobordism
constructions that were inspired by the $2$-handle attachment in Eliashberg's 
symplectic capping argument for contact $3$-manifolds \cite{Eliashberg:cap}.

For the convenience of the reader, we recall from \cite{LisiVanhornWendl1}
the definition of a spinal open
book decomposition for a closed $3$-manifold.
The manifold $M$ is decomposed into two compact regions with
matching boundary, the \defin{spine} $M\spine$ and the
\defin{paper} $M\paper$.
These are equipped with fibrations 
$\pi\spine \colon M\spine \to \Sigma$ and 
$\pi\paper \colon M\paper \to S^1$, where $\Sigma$ is a (possibly disconnected)
oriented surface, each of whose connected components has non-empty boundary,
and the fibers of $\pi\spine$ are assumed connected.
The connected components of the fibers of $\pi\paper$ are called the
\defin{pages}: they also have nonempty boundaries, which are disjoint
unions of fibers of~$\pi\spine$. The spinal open book decomposition is \defin{planar} if the
pages are surfaces of genus 0.

The spinal open book decomposition is then the data
$$
\boldsymbol{\pi} := \Big(\pi\spine : M\spine \to \Sigma,\ 
\pi\paper : M\paper \to S^1\Big).
$$ 
A contact structure is supported by this spinal open book decomposition if
it admits a contact form that restricts to each fiber of $\pi\paper$
as a Liouville form and each fiber of $\pi\spine$ is a closed Reeb
orbit.

Our conditions give that 
the components of $M\spine$ are each $S^1$-bundles over compact oriented
surfaces $\Sigma_1,\ldots,\Sigma_N$ with nonempty boundary. We 
fix a trivialization of each so as to identify the spine with
$$
M\spine = \Sigma_1 \times S^1 \amalg \ldots \amalg \Sigma_N \times S^1.
$$
This choice of trivializations will be referred to in the following as a
\defin{framing} of the spinal open book, and several details will depend
on this choice, but the important point is that it can be fixed in advance,
with no knowledge of the fillings of $(M,\xi)$.

According to the results of \cite{LisiVanhornWendl1}, if $(M,\xi)$ is supported by a
planar spinal open book decomposition, it can only
be strongly fillable if $\boldsymbol{\pi}$ satisfies a condition known as
\emph{symmetry}, which implies that all pages have the same
topology and there exist numbers $k_i \in \NN$ for $i=1,\ldots,N$ such that exactly $k_i$ boundary
components of each page lie in $\Sigma_i \times S^1$.  
In the following, we will assume the spinal open book decomposition satisfies
the hypotheses of 
Theorem~\ref{thm:finCobordism}, so, in particular, we may assume that the spinal
open book decomposition is symmetric.

\begin{remark}
A result for planar spinal open books is also stated in \cite{Kaloti:Stein},
but its proof is framed in terms of Dehn twist factorizations, thus it needs to
assume that the fillings of $(M,\xi)$ are all characterized in terms of 
Lefschetz fibrations.  The latter
is not true for all spinal open books with planar pages, but only for a
special class, satisfying a technical condition known as 
\emph{Lefschetz-amenability} (see \cite{LisiVanhornWendl1}*{\S 1.1}).
More generally, a spinal open book may have the property that its
monodromy permutes boundary components of the pages, in which case
\cite{LisiVanhornWendl2} produces on any filling a foliation by 
$J$-holomorphic curves that can include finitely many
so-called \emph{exotic} fibers,
i.e.~singularities that are different from Lefschetz singular fibers.
The classification problem in these cases requires something more than
an understanding of positive factorizations in the mapping class group.
The following example exhibits a class of contact manifolds to which our
theorem applies, but whose fillings cannot generally be understood in terms
of Lefschetz fibrations with fixed boundary.
\end{remark}

\begin{example}
Suppose $B$ is a closed and connected (but not necessarily orientable)
surface, and $\Gamma \subset B$ is a nonempty multicurve such that
$B \setminus \Gamma$ is orientable.  We say that 
$\Gamma$ \defin{inverts orientations} if for every sufficiently small
open subset $\uU \subset B$ that is divided by $\Gamma$ into two
components $\uU_+$ and $\uU_-$, $\uU$ can be given an orientation that
matches that of $B \setminus \Gamma$ on $\uU_+$ and is the opposite
on~$\uU_-$.  Under this condition, a slight generalization of a well-known
construction of Lutz \cite{Lutz:77} (see \cite{LisiVanhornWendl1}*{\S 1.4}) 
assigns to any $S^1$-bundle $M \to B$ with oriented total space a 
canonical isotopy class of contact structures $\xi_\Gamma$, which are 
tangent to the fibers over $\Gamma$ and positively transverse to the
fibers in $B \setminus \Gamma$ (with respect to their orientations induced
by the orientations of $M$ and $B \setminus \Gamma$).
As shown in \cite{LisiVanhornWendl1}*{\S 1.4}, $(M,\xi_\Gamma)$ admits a
supporting spinal open book with a family of annular pages corresponding
to each component of~$\Gamma$, where the monodromy exchanges boundary
components of the annulus for each component of $\Gamma$ whose normal
bundle is nontrivial.  It is not hard to cook up concrete examples of this
phenomenon where $B$ is the Klein bottle, in which case $(M,\xi_\Gamma)$ can
also be described as a contact parabolic torus bundle (see \cite{LisiVanhornWendl2}).
Theorem~\ref{thm:finCobordism} applies to all contact manifolds of this
type, and thus gives bounds on the geography of their minimal symplectic
fillings.
\end{example}

\begin{figure}
\psfrag{M}{$M$}
\psfrag{W}{$W$}
\psfrag{X}{$X$}
\psfrag{Sigma1xD2}{$\Sigma_1 \times \DD^2$}
\psfrag{Sigma2xD2}{$\Sigma_2 \times \DD^2$}
\psfrag{Sigma1xS1}{$\Sigma_1 \times S^1$}
\psfrag{Sigma2xS1}{$\Sigma_2 \times S^1$}
\psfrag{C1}{$C_1$}
\psfrag{C2}{$C_2$}
\psfrag{pZ}{$\p Z$}
\psfrag{Z}{$Z$}
\includegraphics{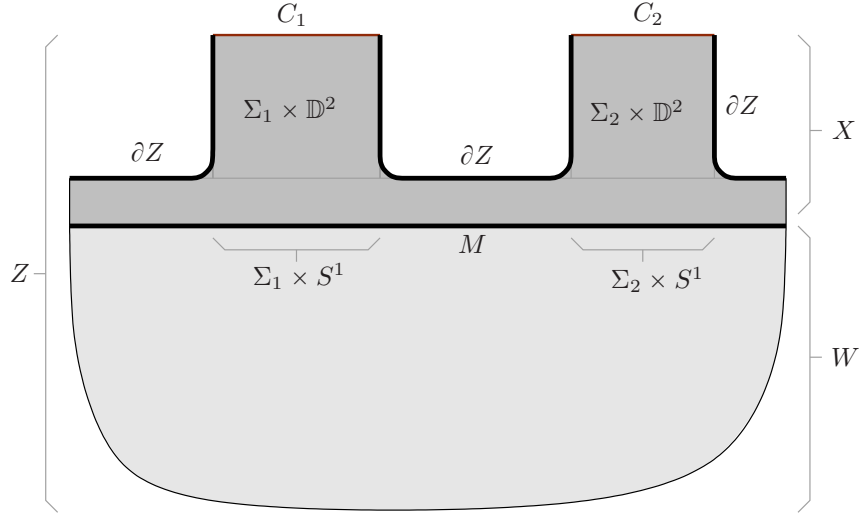}
\caption{\label{fig:spineRemoval}
The manifold $Z = W \cup_M X$ is constructed by stacking a spine removal cobordism $X$
on top of a given filling $W$ of~$M$, where $X$ is formed by attaching a
``handle'' $\Sigma_i \times \DD^2$ on top of the trivial cobordism
$[0,1] \times M$ along every spinal component 
$\Sigma_i \times S^1 \subset M$.  The picture is slightly misleading in that
the subsets $C_i$ for $i=1,2$ do not belong to~$\p Z$; rather, they are the \emph{interior}
codimension~$2$ symplectic submanifolds $\Sigma_i \times \{0\} \subset X$,
i.e.~the co-cores of the handles.}
\end{figure}

Before proving Theorem~\ref{thm:finCobordism} in detail, here is a sketch.
We construct the symplectic manifold $X$ as a spine removal
cobordism that caps every boundary component of the pages (see Figure~\ref{fig:spineRemoval}).  
Planarity implies that the rest of $\p X$ is then a disjoint
union of symplectic $S^2$-fibrations over $S^1$, so for a generic
choice of compatible almost complex structure $J$ making these $S^2$-fibers
complex, the usual holomorphic curve methods (as in 
\cite{McDuff:rationalRuled}) produce a symplectic Lefschetz fibration 
$$
Z := W \cup_M X \stackrel{\Pi}{\longrightarrow} \Sigma_0
$$ 
over some compact oriented surface $\Sigma_0$ with boundary,
where the regular fibers are $J$-holomorphic spheres and the
singular fibers each consist of a pair of transversely
intersecting $J$-holomorphic exceptional spheres.  The co-cores of the spine removal handles
$$
C_1 \cup \ldots \cup C_N \subset X \subset Z
$$
define multisections of this Lefschetz fibration with degrees determined by
the original spinal open book, and they are also symplectic
submanifolds, so they can be arranged to be $J$-holomorphic.  Now since
$\p \Sigma_0 \ne \emptyset$, blowing down an exceptional sphere in
each singular fiber produces a trivial fibration 
$$
\widecheck{Z} \cong \Sigma_0 \times S^2 \stackrel{\widecheck{\Pi}}{\longrightarrow} \Sigma_0,
$$
implying that the topology of $Z$ is
$(\Sigma_0 \times S^2) \# m\overline{\CC P}^2$, where
$m$ is the number of singular fibers.  In order to obtain bounds on
both $\Sigma_0$ and~$m$, we observe first that each
$\Pi|_{C_i} : C_i \to \Sigma_0$ is a branched
cover, so that the Riemann-Hurwitz formula implies a lower bound
on~$\chi(\Sigma_0)$.  Finally, we will find a bound on $m$ by blowing
down $m$ exceptional spheres and considering the resulting multisections
$$
\widecheck{C}_1,\ldots,\widecheck{C}_N \subset \Sigma_0 \times S^2.
$$
Since the
original filling was assumed to be minimal, each exceptional sphere in $Z$
must intersect at least one of the $C_i$'s, thus each blow-down operation
makes some positive contribution to the relative first Chern numbers of the normal
bundles of~$\widecheck{C}_i$.  But since the topology of $\Sigma_0 \times S^2$
is quite simple, we will also be able to show that these Chern numbers depend
only on the framed link $\coprod_{i=1}^N \p \widecheck{C}_i \subset \p \Sigma_0 \times S^2$, 
which (up to some finite ambiguity due to choices of trivialization)
again depends only on the original spinal open book.

The motivating idea in this argument is that the topology of the unknown
filling $W$ can be encoded in the arrangement of positively intersecting
multisections
$\widecheck{C}_1 \cup \ldots \cup \widecheck{C}_N \subset \Sigma_0 \times S^2$,
whose pattern of intersections depends on the intersections of the co-cores
$C_i \subset Z$ with the singular fibers.  Crucially, if we make the right
choices in blowing down~$Z$ to construct~$\widecheck{Z}$, then the
relative homology classes of the $\widecheck{C}_i$ are determined by their 
restrictions to the boundary, and are thus independent of the choice of filling.
Figures~\ref{fig:LefschetzBD1} and~\ref{fig:LefschetzBD2} show an example of how two
slightly different arrangements with the same boundary can correspond to
distinct minimal symplectic fillings of the same contact manifold.

Note that while it is convenient in our argument to assume the co-cores
$C_i \subset X$ are $J$-holomorphic, it is not essential---we use this
assumption mainly in order to ensure that their intersections are positive
and that they have well-defined blow-downs~$\widecheck{C}_i$, but we do not
need any Fredholm or compactness theory for these curves.  That is fortunate,
because in most cases, they live in moduli spaces of negative virtual dimension
and would thus disappear under any nontrivial deformation of~$J$.  There is
one exception: if the original spinal open book is an \emph{ordinary} 
supporting open book in the sense of Giroux \cite{Giroux:ICM}, then
each co-core $C_i$ is a disk that can be completed to a $J$-holomorphic plane 
having index~$0$ and an unobstructed deformation theory.  This provides a
reason to expect strictly more rigidity in the presence of planar open books,
suggesting in particular the following conjecture:

\begin{conju}
Every planar contact $3$-manifold has at most finitely many distinct
deformation classes of minimal symplectic fillings.
\end{conju}

\begin{figure}
\psfrag{C1}{$C_1$}
\psfrag{C2}{$C_2$}
\psfrag{C3}{$C_3$}
\psfrag{C4}{$C_4$}
\psfrag{C1check}{$\widecheck{C}_1$}
\psfrag{C2check}{$\widecheck{C}_2$}
\psfrag{C3check}{$\widecheck{C}_3$}
\psfrag{C4check}{$\widecheck{C}_4$}
\psfrag{Pi}{$\Pi$}
\psfrag{Picheck}{$\widecheck{\Pi}$}
\psfrag{D2}{$\DD^2$}
\psfrag{beta}{$\beta$}
\includegraphics[scale=0.8]{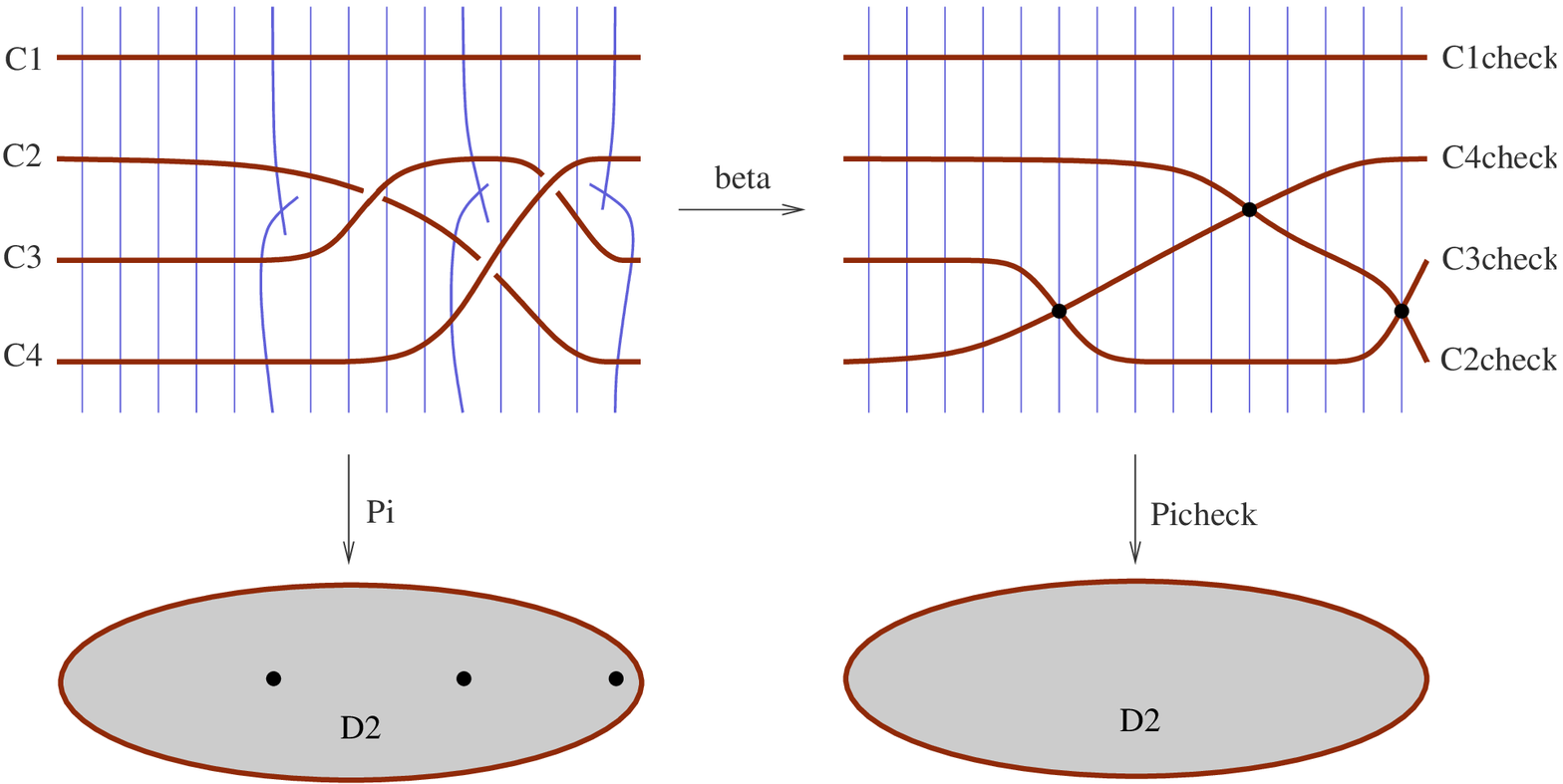}
\caption{\label{fig:LefschetzBD1}
An example with $\Sigma_0 = \DD^2$ for the Lefschetz fibration
$\Pi : Z \to \Sigma_0$ and its blowdown $\widecheck{\Pi} : \widecheck{Z} \to \Sigma_0$,
with four disjoint sections $C_i \subset Z$ and their (no longer disjoint) blowdowns
$\widecheck{C}_i \subset \widecheck{Z}$, defined by composing them with the
blowdown map $\beta : Z \to \widecheck{Z}$.  Here $\widecheck{Z}$ is obtained
from $Z$ by blowing down all exceptional spheres that are disjoint from~$C_1$.}
\ \\
\includegraphics[scale=0.8]{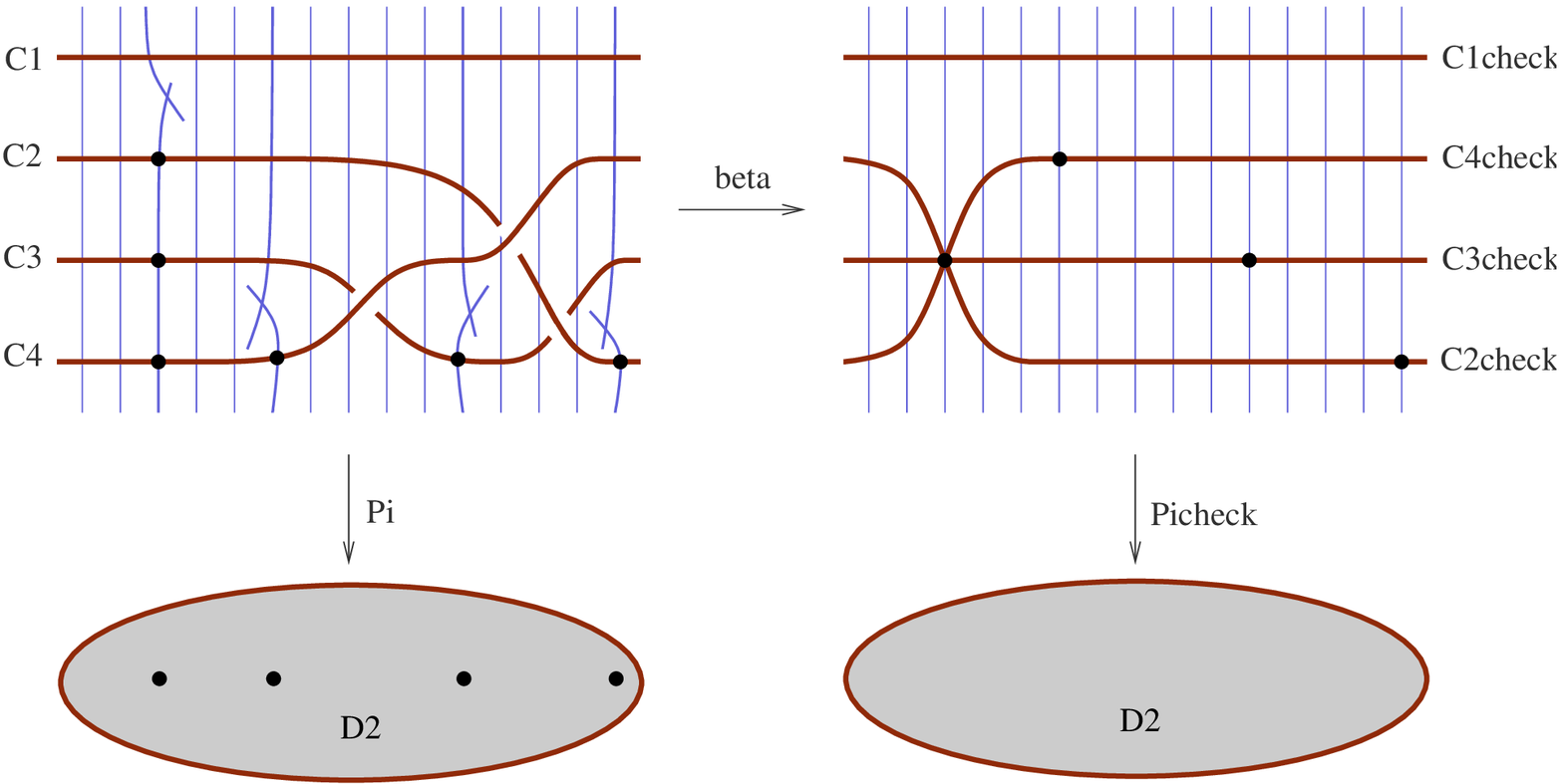}
\caption{\label{fig:LefschetzBD2}
Another example of the Lefschetz fibration on $Z = W \cup_M X$ and its blowdown,
but with $W$ chosen to be a different minimal filling of the same contact
manifold $M$ as in Figure~\ref{fig:LefschetzBD1}.  In particular, the blown-down sections
$\widecheck{C}_i \subset \widecheck{Z}$ are homologous to those in
Figure~\ref{fig:LefschetzBD1} and restrict to the same framed link at the boundary.
One can recover Lefschetz fibrations of both fillings by deleting neighborhoods 
of the sections $C_i \subset Z$ in both figures: in this
example, they are related to each other by the lantern relation, so the
two pictures represent the two Stein fillings of $L(4,1)$ described originally
by McDuff \cite{McDuff:rationalRuled}.}
\end{figure}

\subsection*{Acknowledgements}

This paper emerged as a side project from our long-running collaboration with
Jeremy Van Horn-Morris, and we are deeply indebted to him for many
enlightening discussions.  The second author would also like to thank
Tom Mark, who originally suggested the idea of using the symplectic co-cores
of spine removal handles to extract information about symplectic fillings.

\section{The proof}

\subsection{Spine removal}

Assume for the rest of this paper that $(M,\xi)$ satisfies the hypotheses
of Theorem~\ref{thm:finCobordism}. In particular, $\xi$ is supported by the
symmetric, planar spinal open book decomposition 
$$
\boldsymbol{\pi} := \Big(\pi\spine : M\spine \to \Sigma,\ 
\pi\paper : M\paper \to S^1\Big).
$$
Fix a framing of the spinal open book, which then identifies each component of
the spine with $\Sigma_i \times S^1$.  Since the spinal open book
decomposition is symmetric, there are integers $k_i$ that give the
incidence of any page with the spine component $\Sigma_i \times S^1$.

Denote by
$$
X \cong \left([0,1] \times M \right) \cup_{\{1\} \times M\spine} \coprod_{i=1}^N \Sigma_i \times \DD^2
$$
the cobordism obtained by performing spine removal surgery on every
component of $M\spine$, attaching $\Sigma_i \times \DD^2$ along its
boundary in the obvious way to each spinal component $\Sigma_i\times S^1 \subset \{1\} \times M$
and then smoothing the corners (see Figure~\ref{fig:spineRemoval}).  
According to \cite{LisiVanhornWendl1}, this cobordism carries
a symplectic structure $\omega_X$ such that we can write
$$
\p X = (- \p_- X) \amalg \p_+ X,
$$
where $\p_- X = M$ is concave with induced contact structure~$\xi$, and
$\p_+ X$ is a disjoint union of symplectic fibrations over~$S^1$, one for each connected
component of~$M\paper$, whose fibers are closed surfaces obtained by capping off all
boundary components of the pages with disks.  By the assumption that the spinal
open book is planar, these are spheres.  Since
$\Diff_+(S^2) \simeq \SO(3)$ is connected, these fibrations are all trivial
and are thus diffeomorphic to $S^1 \times S^2$, and each of the fibers has exactly $k_i$
transverse and positive intersections with the symplectic co-cores
$$
C_i := \Sigma_i \times \{0\} \subset \Sigma_i \times \DD^2 \subset X
$$
for $i=1,\ldots,N$.  The boundaries 
$$
\p C_i = \p\Sigma_i \times \{0\} \subset \p_+ X
$$ 
form a link transverse to the sphere fibers, which comes with a natural
framing dependent on our original choice of framing for~$\boldsymbol{\pi}$.

\subsection{Lefschetz fibration and multisections}

Now suppose $(W,\omega_W)$ is a minimal strong filling of $(M,\xi)$.
After a rescaling and deformation of $\omega_W$ near~$\p W$, we glue
symplectically to form an enlarged symplectic manifold
$$
(Z,\omega) := (W,\omega_W) \cup_M (X,\omega_X)
$$
with boundary $\p Z = \p_+ X$.  Choose an $\omega$-compatible almost complex
structure $J$ on $Z$ that makes both the co-cores $C_1,\ldots,C_N$ and
the $S^2$-fibers on $\p_+ X$ into $J$-holomorphic curves and is generic
everywhere else.  Then by standard arguments as in 
\cites{McDuff:rationalRuled,Wendl:rationalRuled},
the compactified moduli space of $J$-holomorphic spheres
homotopic to the fibers on $\p_+ X$ forms the fibers of a Lefschetz fibration
$$
\Pi : Z \to \Sigma_0,
$$
where $\Sigma_0$ is a compact oriented surface (homeomorphic
to the compactified moduli space) whose number of boundary components is
the number of connected components of~$M\paper$.  The singular fibers
of $\Pi : Z \to \Sigma_0$ each have two irreducible components, both 
exceptional spheres which intersect each other once transversely.
Let $[F] \in H_2(Z)$ denote the homology class of the fibers and
$[C_i] \in H_2(Z,\p C_i)$ the relative homology class of the co-core
$C_i \subset X \subset Z$ for each $i=1,\ldots,N$; the intersection products
$[F] \cdot [C_i] \in \ZZ$ are then well defined and satisfy
$$
[F] \cdot [C_i] = k_i.
$$
By positivity of intersections, it follows that 
$C_i$ intersects each fiber at most $k_i$ times, with equality for the fibers
to which it is transverse.  A standard genericity argument
(cf.~\cite{CieliebakMohnke:transversality}*{Prop.~9.1(b)}) implies:

\begin{lemma}
\label{lemma:genericIntersections}
The following holds for generic choices of $J$ satisfying the conditions 
described above.  All singular fibers of
$\Pi : Z \to \Sigma_0$ are transverse to each co-core $C_i$, with
intersections occurring only at regular points, while the 
collection of all regular fibers that intersect one of the co-cores
nontransversely is finite.  Moreover, every such fiber $F \subset Z$ has
at most one non-transverse intersection with any of the co-cores, and it
has local intersection index~$2$.  \qed
\end{lemma}

We will need a slightly more precise picture of these curves in certain
regions.

\begin{lemma}
\label{lemma:coordinates}
Suppose $E_1,\ldots,E_m \subset Z$ is a collection of pairwise disjoint
exceptional spheres that are each irreducible components of singular fibers
for~$\Pi$.  Then after smooth and $C^0$-small deformations of $J$, the 
Lefschetz fibration, and the
co-cores $C_1,\ldots,C_N$ such that the co-cores and fibers remain
$J$-holomorphic, the following can be assumed without loss of generality:
\begin{enumerate}
\item For each $i=1,\ldots,m$, $J$ is integrable on a neighborhood of~$E_i$,
which is biholomorphically equivalent to a neighborhood of the zero-section in the
tautological line bundle over~$\CP^1$.
\item For each $i=1,\ldots,N$ and each point $z_0 \in \Sigma_0$ such that
the fiber over $z_0$ has a tangential intersection with~$C_i$,
we can identify a neighborhood 
$\uU \subset \Sigma_0$ of $z_0$ with $\DD^2$ and identify
$\Pi^{-1}(\uU)$ with $\DD^2 \times S^2$ where $S^2 := \CC \cup \{\infty\}$,
such that on this neighborhood, $J = i \oplus i$, $\Pi(z,w) = z$, and
$(C_1 \cup \ldots \cup C_N) \cap \Pi^{-1}(\uU)$ is the disjoint union of one
surface of the form
$$
\left\{ (z^2,az) \in \DD^2 \times S^2 \ \big|\ z \in \DD^2 \right\}, \qquad a \in \CC \setminus \{0\}
$$
with a finite collection of other surfaces of the form
$$
\left\{ (z,b) \in \DD^2 \times S^2\ \big|\ z \in \DD^2 \right\} \qquad b \in \CC.
$$
\end{enumerate}
\end{lemma}
\begin{proof}
Two preliminary remarks: first, the $J$-holomorphic fibers of the
Lefschetz fibration will deform smoothly under any generic deformation of~$J$.
Indeed, deformations are unobstructed for all $J$ since the curves satisfy
the automatic transversality criterion of \cite{HoferLizanSikorav}, and for
index reasons, the only danger of new bubbling under these deformations
would come from index~$-1$ curves when $J$ becomes nongeneric, but closed
holomorphic curves can only have even index (cf.~\cite{Wendl:rationalRuled}*{Remark~2.20}).
Second, the deformation of $J$ does not need to be generic everywhere, as 
it suffices to have genericity only in some open region that intersects all
of the curves we are concerned about.

Now, denote by $\widetilde{\CC}^2 \to \CP^1$ the tautological line bundle,
which can be obtained by performing a complex blowup on $\CC^2$ at the origin,
so let $i$ denote its standard complex structure.
Since the co-cores $C_i$ intersect each of the exceptional spheres $E_j$
transversely, we can identify the complex normal bundle of $E_j$ with a
$J$-invariant subbundle of $TZ|_{E_j}$ that is tangent to each~$C_i$, and
feeding this into the tubular neighborhood theorem then identifies a
neighborhood of $E_j$ with a neighborhood of the zero-section $\CP^1 \subset \widetilde{\CC}^2$
such that $J$ matches $i$ along $E_j = \CP^1$ and each $C_i$ is tangent to
fibers of $\widetilde{\CC}^2 \to \CP^1$ at the zero-section.  We can then make a $C^1$-small
deformation of each $C_i$ near $E_j$ so that it matches fibers \emph{precisely}
in some smaller neighborhood of~$E_j$,
and a corresponding $C^0$-small deformation of $J$ to make it match $i$
in this smaller neighborhood; since the change in $C_i$ was $C^1$-small, we
can assume it is still symplectic and thus adjust $J$ correspondingly away
from $E_j$ to make sure that $C_i$ is $J$-holomorphic.  Notice that $E_j$
remains holomorphic throughout this deforation, which we can also assume is
supported in a tubular neighborhood of~$E_j$; then since $[E_j] \cdot [E_j] = -1$,
positivity of intersections implies that $E_j$ is always the only closed
holomorphic curve that is contained fully in this neighborhood, so that all
other fibers (or irreducible components of fibers) of $\Pi$ pass through
regions in which $J$ can still be assumed generic, hence they deform smoothly
as indicated in the previous paragraph.

The changes near $E_j$ can be assumed to have no effect on any fiber $F := \Pi^{-1}(z_0)$
that intersects some co-core $C_i$ non-transversely, since $F$ is a regular
fiber.  To understand the neighborhood of~$F$,
identify its (necessarily trivial) complex normal bundle with a $J$-invariant subbundle of $TZ|_F$
that is tangent to every co-core $C_i$ except at the unique point where one
of them is tangent to~$F$, and use this to identify the
neighborhood of $F$ with $\DD^2 \times S^2$ such that $F = \{0\} \times S^2$
and $J|_F = i \oplus i$.  One can also arrange this so that the surfaces
$\{z\} \times S^2$ are all fibers of~$\Pi$, and after suitably reparametrizing~$S^2$,
we can assume all transverse intersections of $C_1 \cup \ldots \cup C_N$ with $F$
occur at points $(0,b) \in \DD^2 \times S^2$ with $b \ne \infty$, while
the tangential intersection occurs at~$(0,0)$.  In light of our choice of
normal subbundle, the surfaces $C_j$ are tangent to surfaces of the form
$\DD^2 \times \{\operatorname{const}\}$ whenever they intersect $F$ transversely.
We can now modify $J$ near $F$ without changing it in directions tangent to
the fibers such that, after a $C^0$-small deformation, $J = i \oplus i$
near~$F$; after a similar $C^1$-small adjustment to the co-cores $C_j$ near
their tranverse intersections with~$F$, these will also have the form
$\DD^2 \times \{b\}$ near those intersections, and $J$ can still be adjusted
outside this small neighborhood to make the modified $C_j$ holomorphic.
Near the tangential intersection of $C_i$ and~$F$, Taylor's theorem and the
assumption that $C_i$ intersects $F$ with local index $2$ implies that $C_i$
can be parametrized by an embedded holomorphic curve of the form
$$
u : \DD^2 \hookrightarrow \DD^2 \times S^2 : z \mapsto
(z^2 + R(z), a z)
$$
for some $a \in \CC \setminus \{0\}$ and a smooth remainder function
$R(z) \in \CC$ satisfying $\lim_{z \to 0} \frac{R(z)}{|z|^2} = 0$.
Choosing $\epsilon > 0$ small and a smooth cutoff function
$\beta : [0,\infty) \to [0,1]$ with $\beta(s) = 1$ for $s \ge 1$ and $\beta(s)=0$
near $s=0$, we modify $u$ to an embedding of the form
$$
u_\epsilon(z) := (z^2 + \beta(|z|/\epsilon) R(z) , a z),
$$
and modify $C_i$ near $(0,0)$ by defining it as the image of~$u_\epsilon$.
Modifying $J$ to match $i \oplus i$ near~$F$ now makes  $u_\epsilon$ $J$-holomorphic
near the intersection, and for $\epsilon > 0$ sufficiently small it still
traces out a symplectic submanifold $C^1$-close to the original~$C_i$,
thus $J$ can be modified further outside the neighborhood of $(0,0)$ to make
the new $C_i$ everywhere $J$-holomorphic.  The latter modification can be
assumed to take place in a small neighborhood that contains no closed
holomorphic curves, thus we can assume all fibers passing through that
neighborhood also pass through regions where $J$ is generic, and they therefore
survive the deformation.
\end{proof}

\begin{cor}
\label{cor:isABranchedCover}
For each $i=1,\ldots,N$, the map 
$$
\varphi_i := \Pi|_{C_i} : C_i \to \Sigma_0
$$
is a
branched cover of degree~$k_i$.  Moreover, its branch points are all simple
(i.e.~of order~$2$), and for any $i,j = 1,\ldots,N$, 
any two distinct branch points of $\varphi_i : C_i \to \Sigma_0$ and
$\varphi_j : C_j \to \Sigma_0$ have distinct images in~$\Sigma_0$,
all of which are regular values of $\Pi : Z \to \Sigma_0$.  \qed
\end{cor}

We shall refer to the surfaces $C_i \subset Z$ as \defin{multisections}
of $\Pi : Z \to \Sigma_0$ with degree~$k_i$.  The multisection $C_i$ is an honest
section if and only if $k_i=1$.  
Let $b_i$ be the number of branch points of the branched covering 
$\varphi_i \colon C_i \to \Sigma_0$ from Corollary~\ref{cor:isABranchedCover}. 
Then, the Riemann-Hurwitz formula gives
\[
    k_i \chi(\Sigma_0) = \chi(C_i) + b_i
\]
since each branch point is simple. We then have
\begin{equation}
\label{eqn:boundSigma0}
\chi(\Sigma_0) \ge \frac{1}{k_i} \chi(C_i)
\quad\text{ for each $i=1,\ldots,N$}.
\end{equation}
Recall that the number of boundary components of $\Sigma_0$ is given by the
number of connected components of the paper $M\paper$. 
This therefore gives an upper bound on the genus of $\Sigma_0$ that is determined
by the spinal open book $\boldsymbol{\pi}$ and hence is independent of the
choice of filling~$W$.  The topological type of $\Sigma_0$ thus belongs
to a finite list of possibilities determined by $(M,\xi)$ and~$\boldsymbol{\pi}$.

Let $m \ge 0$ denote the number of singular fibers in the
Lefschetz fibration $\Pi : Z \to \Sigma_0$.
Blowing down one exceptional sphere in each of these fibers then gives
a smooth $S^2$-fibration over $\Sigma_0$, which is necessarily
trivial since $\p\Sigma_0 \ne \emptyset$, hence
$$
Z \cong (\Sigma_0 \times S^2) \# m\overline{\CC P}^2.
$$
The remaining task is thus to establish an a priori upper bound on the
number of singular fibers~$m$.

\begin{lemma}
\label{lemma:notMinimal}
Every symplectic exceptional sphere in $Z$ intersects at least one of
the co-cores $C_1,\ldots,C_N$.
\end{lemma}
\begin{proof}
If not, then one can remove tubular neighborhoods of the co-cores and
find a symplectic filling that is symplectic deformation equivalent to
$(W,\omega)$ but contains an exceptional sphere, contradicting the
assumption that $(W,\omega)$ is minimal.
\end{proof}

The argument for bounding $m$ is slightly more straightforward if we can assume
one of the numbers $k_i$ is~$1$, so let us consider the case $k_1 = 1$
before tackling the general situation.  

\subsection{The case \texorpdfstring{$k_1=1$}{k1 = 1}}

Under this assumption, $C_1$ is a section of
$\Pi : Z \to \Sigma_0$, so in particular, every singular fiber consists of
exactly one exceptional sphere that intersects $C_1$ once and one that is
disjoint from~$C_1$.  We shall denote the spheres of the latter type by
$$
E_1,\ldots,E_m \subset Z
$$
and note that they are in one-to-one correspondence with the singular
fibers of $\Pi : Z \to \Sigma_0$.  Applying Lemma~\ref{lemma:coordinates}
to deform $J$ so that it becomes integrable near $E_1 \cup \ldots \cup E_m$,
define $\widecheck{Z}$ to be the symplectic manifold obtained from $Z$ by performing
a blow-down operation on each $E_1,\ldots,E_m$.  This symplectic manifold
inherits a compatible almost complex structure
$\widecheck{J}$ such that the blowdown map
$$
\beta : (Z,J) \to (\widecheck{Z},\widecheck{J}),
$$
defined as the identity outside of $E_1 \cup \ldots \cup E_m$ but collapsing
each $E_i$ to a point, is pseudoholomorphic.
Since all singular fibers of $\Pi : Z \to \Sigma_0$ are in the interior,
the boundary of $\widecheck{Z}$ is still $\p Z$ and thus contains the same
framed link $\coprod_{i=1}^N \p C_i \subset \widecheck{Z}$.
The Lefschetz fibration $\Pi : Z \to \Sigma_0$ now becomes a smooth symplectic
$S^2$-fibration 
$$
\widecheck{\Pi} : \widecheck{Z} \to \Sigma_0
$$
with $\widecheck{J}$-holomorphic fibers, and
the co-cores
$C_i$ for $i=1,\ldots,N$ project through $\beta : Z \to \widecheck{Z}$ 
to immersed (but possibly non-injective) $\widecheck{J}$-holomorphic curves
$$
\widecheck{C}_i \looparrowright \widecheck{Z}
$$
with boundary $\p\widecheck{C}_i = \p C_i \subset \p\widecheck{Z}$.
These curves are immersed since
each $C_i$ is transverse to every singular fiber by 
Lemma~\ref{lemma:genericIntersections}, but they have transverse 
self-intersections whenever $C_i$ intersects one 
of the $E_j$ more than once.  They are again multisections of 
$\widecheck{\Pi} : \widecheck{Z} \to \Sigma_0$ with degree~$k_i$.  
Notice that since $C_i$ and $C_j$ might both intersect some~$E_k$,
$\widecheck{C}_i$ and 
$\widecheck{C}_j$ 
might intersect. The co-cores $C_i$ and $C_j$ intersect~$E_k$
in different points, however, so the intersection of 
$\widecheck{C}_i$ and 
$\widecheck{C}_j$ 
will be transverse and positive.

If $\tau$ denotes the framing of
$\p C_i \subset \p Z$ we fixed previously, then by construction,
$\tau$ extends to a global trivialization of the complex normal bundle 
$N_{C_i} \subset TZ|_{C_i}$ of $C_i$, so that its relative first Chern number
is $c_1^\tau(N_{C_i}) = 0$.  After blowing down, each of these Chern numbers
will in general be larger, as can be measured by counting the intersections
of $C_i$ with a small perturbation of itself: a
positive transverse intersection in the blowdown is forced wherever $C_i$ and its perturbation
pass through one of the~$E_j$, proving
\begin{equation}
\label{eqn:cocoreNormal}
c_1^\tau(N_{\widecheck{C}_i}) = \sum_{j=1}^m [C_i] \cdot [E_j] \ge 0.
\end{equation}
Since we chose 
$E_1,\ldots,E_m$ to be disjoint from~$C_1$,
we have $c_1^\tau(N_{\widecheck{C}_1}) = 0$.
Note that $\widecheck{C}_1$ also does not intersect any of
$\widecheck{C}_2,\ldots,\widecheck{C}_N$.  

We claim that for each $i=1,\ldots,N$, 
the left hand side of \eqref{eqn:cocoreNormal} satisfies a bound that is 
independent of the choice of minimal filling for~$(M,\xi)$.
To see this, note 
first that we can write
$$
c_1^\tau(N_{\widecheck{C}_i}) = c_1^\tau(\widecheck{C}_i) - \chi(C_i),
$$
where $c_1^\tau(\widecheck{C}_i)$ is an abbreviation for the relative first Chern number
of the complex vector bundle $T\widecheck{Z}$ pulled back along the
immersion $\widecheck{C}_i \looparrowright \widecheck{Z}$, relative to the
obvious trivialization induced by $\tau$ at the boundary.  In particular,
$\chi(C_i)$ is determined by the spinal open book, and
$c_1^\tau(\widecheck{C}_i)$ depends only on the relative homology class
$[\widecheck{C}_i] \in H_2(\widecheck{Z} , \p C_i)$.  

To understand the latter, observe first that
there are finitely many homotopy classes of trivializations of the
oriented $S^2$-bundle 
$$
\p Z = \p\widecheck{Z} \stackrel{\widecheck{\Pi}}{\longrightarrow} \p\Sigma_0,
$$
corresponding to choices of elements in 
$\pi_1(\Diff_+(S^2)) \cong \pi_1(\SO(3)) \cong \ZZ_2$ for each component
of~$\p\Sigma_0$.  Each of these trivializations identifies the framed link
$\coprod_{i=1}^N \p\widecheck{C}_i \subset \p\widecheck{Z} = \p Z$ with some
isotopy class of framed links in $\p\Sigma_0 \times S^2$ transverse to the
$S^2$-fibers, producing a finite list of such framed links that depends
only on the original framed spinal open book and not on the filling~$W$.

Note that $\Sigma_0$ is a surface with boundary and retracts to its $1$-skeleton. The fibration
$\widecheck{\Pi} : \widecheck{Z} \to \Sigma_0$ is thus globally trivializable.
After choosing such a trivialization, 
we may identify 
$$
\widecheck{Z} = \Sigma_0 \times S^2 \stackrel{\widecheck{\Pi}}{\longrightarrow}
\Sigma_0 : (z,w) \mapsto z,
$$
where the isotopy class of the framed link $\coprod_{i=1}^N \p\widecheck{C}_i \subset
\p\Sigma_0 \times S^2$ must belong to the aforementioned finite list, which is
independent of the filling. 

We now claim that the isotopy
class of this framed link uniquely determines the relative homology classes
$[\widecheck{C}_i]$ for $i=1,\ldots,N$.  Indeed, let $[\widecheck{C}_i] \cdot_\tau [\widecheck{C}_i] \in \ZZ$
denote the \defin{relative self-intersection number} of $[\widecheck{C}_i]$, defined
by counting the signed intersections of $\widecheck{C}_i$ with a small generic perturbation
of itself that is pushed in the direction of $\tau$ at the boundary
(cf.~\cites{Hutchings:index,Siefring:intersection}).  Then
$$
[\widecheck{C}_1] \cdot_\tau [\widecheck{C}_1] = [C_1] \cdot_\tau [C_1] = 0
$$
by construction.  If $C' \looparrowright \widecheck{Z}$ is then another
properly immersed surface with $\p C' = \p\widecheck{C}_1$, the relative
homology classes of $C'$ and $\widecheck{C}_1$ differ by some absolute
class in $H_2(\Sigma_0 \times S^2) \cong \ZZ$, which is generated by the
fiber class~$[F]$, thus $[C'] = [\widecheck{C}_1] + \ell [F]$ for some
$\ell \in \ZZ$, and we have
$$
[C'] \cdot_\tau [C'] = [\widecheck{C}_1] \cdot_\tau [\widecheck{C}_1] +
2 \ell [\widecheck{C}_1] \cdot [F] + \ell^2 [F] \cdot [F] = 2 \ell
$$
since $[\widecheck{C}_1] \cdot [F] = k_1 = 1$ and $[F] \cdot [F] = 0$.
It follows that $C'$ can only have relative self-intersection zero if it is
homologous to~$\widecheck{C}_1$. Thus 
$[\widecheck{C}_1] \in H_2(\Sigma_0 \times S^2,\p C_1)$ is uniquely determined.
For $\widecheck{C}_i$ with $i=2,\ldots,N$, the fact that 
$$
[\widecheck{C}_1] \cdot [\widecheck{C}_i] = 0
$$
and $[\widecheck{C}_1] \cdot [F] \ne 0$ similarly implies that 
$[\widecheck{C}_i] \in H_2(\Sigma_0 \times S^2,\p C_i)$ is uniquely determined.

From this, it follows that for $i=1,\ldots,N$, there exist integers $m_i \in \ZZ$
depending only on the framed spinal open book $\boldsymbol{\pi}$ such that
$$
\sum_{j=1}^m [C_i] \cdot [E_j] \le m_i.
$$
At the same time, Lemma~\ref{lemma:notMinimal} implies
$$
\sum_{i=1}^N [C_i] \cdot [E_j] \ge 1
$$
for each $j=1,\ldots,m$, thus
$$
m \le \sum_{j=1}^m \sum_{i=1}^N [C_i] \cdot [E_j] \le \sum_{i=1}^N m_i,
$$
and this bound depends only on~$\boldsymbol{\pi}$ with its chosen framing.  
Since $Z \cong (\Sigma_0 \times S^2) \# m\overline{\CC P}^2$,
this concludes the proof for the case $k_1 = 1$.

\subsection{The general case}

The above strategy fails if none of the $k_i = [C_i] \cdot [F]$ equal~$1$ since it could
then happen that for every singular fiber of $\Pi : Z \to \Sigma_0$,
both exceptional spheres intersect every co-core~$C_i$.  We can
deal with this by replacing $Z$ by a branched cover, defined essentially as
the pullback of the Lefschetz fibration $\Pi : Z \to \Sigma_0$ via the
map $\varphi_1 : C_1 \to \Sigma_0$.  

Let us begin with a
construction at the boundary that requires no knowledge of the filling. Recall
that $\p Z$ and $\p \Sigma_0$ depend only on the spinal open book decomposition
$\boldsymbol{\pi}$, and do not depend on the filling (though $Z$ and $\Sigma_0$ do). 
Since $\p C_1$ is transverse to the fibers of $\p Z \stackrel{\Pi}{\to} \p\Sigma_0$,
the map $\p C_1 \stackrel{\varphi_1}{\to} \p\Sigma_0$ is a smooth
$k_1$-fold covering map and
there exists a smooth closed $3$-manifold\footnote{The notation $\p Z'$ is
chosen because it will turn out to be the boundary of something constructed
further below, but the definition of $\p Z'$ itself does not require this
knowledge.}
$$
\p Z' := \left\{ (z,x) \in \p C_1 \times \p Z\ \Big|\ \varphi_1(z) = \Pi(x) \right\},
$$
which admits a smooth oriented $S^2$-fibration
$$
\Pi' : \p Z' \to \p C_1 : (z,x) \mapsto z
$$
and a $k_1$-fold covering map
$$
\Phi : \p Z' \to \p Z : (z,x) \mapsto x
$$
such that $\varphi_1 \circ \Pi' = \Pi \circ \Phi$.  The framed link
$\coprod_{i=1}^N \p C_i \subset \p Z$ then gives rise to a new framed link
consisting of the disjoint union of
$$
\p C_i' := \Phi^{-1}(\p C_i) \subset \p Z',\qquad i=1,\ldots,N.
$$
These links are similarly transverse to the $S^2$-fibers, and $\p C_1'$ contains a
distinguished component given by the \emph{tautological section}
$$
\p \sigma_1 := \{ (z,z) \in \p Z' \ |\ z \in \p C_1 \}
$$
which is a lift of $\p C_1 \subset \p Z$ to the cover.  All of this depends
only on the framed link $\coprod_{i=1}^N \p C_i$ in~$\p Z$, and thus on the
original framed spinal open book, but not on the filling~$W$.

Next we extend the covering map $\p Z' \to \p Z$ to a branched cover $Z' \to Z$.
The construction is completely analogous to that of the previous paragraph, 
except that it does depend on the filling: recall first that by
Corollary~\ref{cor:isABranchedCover}, the branch points of 
$\varphi_1 : C_1 \to \Sigma_0$ occur in regular fibers, implying that
$\varphi_1$ is transverse to~$\Pi$ and the set
$$
Z' := \left\{ (z,x) \in C_1 \times Z\ \Big|\ \varphi_1(z) = \Pi(x) \right\}
$$
is therefore a smooth compact $4$-manifold with boundary~$\p Z'$, while
$$
\Pi' : Z' \to C_1 : (z,x) \mapsto z
$$
is a Lefschetz fibration and
$$
\Phi : Z' \to Z : (z,x) \mapsto x
$$
is a $k_1$-fold branched cover satisfying $\varphi_1 \circ \Pi' = \Pi \circ \Phi$.  
The regular fibers of
$\Pi' : Z' \to C_1$ are again spheres, and $\Phi$ is locally $2$-to-$1$
near a branching locus that consists of a finite collection of regular fibers
corresponding to the branch points of~$\varphi_1$.
Since singular fibers are disjoint from the branching locus, 
$\Pi' : Z' \to C_1$ has exactly $k_1 m$ singular fibers.

For each $i=1,\ldots,N$, let
$$
C'_i := \Phi^{-1}(C_i) \subset Z'.
$$
For $i=2,\ldots,N$, $C'_i \subset Z'$ is a submanifold since 
(by Corollary~\ref{cor:isABranchedCover}) the branch points of
$\varphi_i : C_i \to \Sigma_0$ occur outside the branching locus of $\Phi$,
hence $\Phi$ is transverse to~$C_i$.  
The situation is
slightly different for $i=1$ since $\Phi$ is not transverse to~$C_1$, and
$\Phi^{-1}(C_1)$ is the set of all pairs $(z,x) \in C_1 \times C_1$ such
that $z$ and $x$ belong to the same fiber of~$\Pi$.  This contains the
tautological section of $\Pi' : Z' \to C_1$ defined by
$$
\sigma_1 := \{ (z,z) \in Z' \ |\ z \in C_1 \}.
$$

\begin{lemma}
\label{lemma:C1prime}
We have $C_1' = \sigma_1 \cup C_1''$, where
$C_1''$ is a (possibly disconnected) smooth submanifold of $Z'$ that has finitely many
intersections with $\sigma_1$, all transverse and positive, occurring
at each of the points $(z,z) \in \sigma_1$ for the branch points $z \in C_1$ of
$\varphi_1 : C_1 \to \Sigma_0$.  Moreover, the map
$$
\varphi_1' := \Pi'|_{C_1''} : C_1'' \to C_1
$$
is a branched cover of degree $k_1 - 1$.
\end{lemma}
\begin{proof}
As a set, we define $C_1''$ to be the closure of $C_1' \setminus \sigma_1$:
$$
C_1'' := \overline{C_1' \setminus \sigma_1} \subset Z'.
$$
Recall that $Z'$ was defined as a smooth submanifold of $C_1 \times Z$, and $C_1'$ is the
intersection of this with the submanifold $C_1 \times C_1 \subset C_1 \times Z$.
We claim that for any $(z,x) \in Z' \cap (C_1 \times C_1)$ with the property
that $z$ and $x$ are not both branch points of $\varphi_1 : C_1 \to \Sigma_0$,
this intersection of submanifolds in $C_1 \times Z$ is transverse.
This is equivalent to the claim that the tangent space
$$
T_{(z,x)}Z' = \left\{ (X,Y) \in T_z C_1 \oplus T_x Z\ \big|\ 
T\varphi_1(X) - T\Pi(Y) = 0 \right\}
$$
contains elements of the form $(X,Y)$ for arbitrary vectors
$Y$ spanning some subspace of $T_x Z$ transverse to $T_x C_1$.  
If $z$ is a regular point of $\varphi_1$, then $T_z \varphi_1 :
T_z C_1 \to T_{\varphi_1(x)} \Sigma_0$ is an isomorphism, so
for any $Y \in T_x Z$ there is a unique $X \in T_z C_1$ satisfying 
$T\varphi_1(X) = T\Pi(Y)$, which gives $(X,Y) \in T_{(z,x)}Z'$.
The alternative is that $T_z \varphi_1$ vanishes but $x$ is a regular point
of~$\varphi_1$, so
$T_{(z,x)}Z' = T_z C_1 \oplus \ker T_x \Pi$, where $\ker T_x\Pi$
is the tangent space to the fiber at~$x$.  The regularity of $\varphi_1$ at
$x$ then implies that
$\ker T_x\Pi$ is transverse to $T_x C_1$ and thus proves the claim.

Observe that by Corollary~\ref{cor:isABranchedCover}, any
$(z,x) \in Z' \cap (C_1 \times C_1)$ with $z \ne x$ has the property that
$z$ and $x$ cannot both be branch points of~$\varphi_1$, since
$\varphi_1(z) = \varphi_1(x)$.  We have thus proved that the only
possible singularities of $C_1'$ are at points
$(z,z) \in \sigma_1$ such that $z$ is a branch point of~$\varphi_1$.

We now analyze a neighborhood of any such point, using the local 
model from Lemma~\ref{lemma:coordinates}. Locally, we can assume
$\Sigma_0 = \DD^2 \subset \CC$, $Z = \DD^2 \times S^2$, $\Pi(\zeta,w) = \zeta$
and $C_1 \subset Z$ is parametrized by an embedding of the form
\begin{equation}
\label{eqn:branchParam}
\DD^2 \hookrightarrow \DD^2 \times S^2 : \zeta \mapsto (\zeta^2,a \zeta)
\end{equation}
for some $a \in \CC \setminus \{0\}$.  Since there are no singular
fibers in this local picture, $\Pi' : Z' \to C_1$ is actually just the
pullback of $\Pi : Z \to \Sigma_0$ through $\varphi_1 : C_1 \to \Sigma_0$,
and the parametrization for $C_1$ above then provides a smooth trivialization
of~$\Pi'$ identifying $((\zeta^2,a \zeta),(\zeta^2,x)) \in C_1 \times Z$
with $(\zeta,x) \in \DD^2 \times S^2$.  Under this identification, we can write
$$
Z' = \DD^2 \times S^2,
$$
with $\Phi(\zeta,x) = (\zeta^2,x)$ and $\Pi'(\zeta,x) = (\zeta^2,a \zeta)$,
and the tautological section is now parametrized by
$$
\sigma_1 : \DD^2 \hookrightarrow \DD^2 \times S^1 : \zeta \mapsto
(\zeta,a \zeta).
$$
We then have
$$
C_1' = \left\{ (\zeta,x) \in \DD^2 \times S^2 \ \big|\ (\zeta^2,x) \in C_1 \right\} =
\left\{ (\pm \zeta,a \zeta) \in \DD^2 \times S^2 \ \big|\ \zeta \in \DD^2 \right\},
$$
which can be written as the union of two transversely and positively intersecting
submanifolds
$$
\{ (\zeta,a \zeta) \} \cup \{ (\zeta,-a \zeta) \},
$$
where the first is the tautological section~$\sigma_1$.
It thus follows that $C_1''$ is a smooth submanifold of~$Z'$.

Finally, the coordinate description above implies that the map 
$\varphi_1' : C_1'' \to C_1$ is a local diffeomorphism near $C_1'' \cap \sigma_1$,
and it is clearly also a local diffeomorphism near all
points of the form $(z,x) \in C_1''$ such that $z$ and $x$ are
regular points of~$\varphi_1$.  
If on the other hand $(z,x) \in C_1''$ 
where $x$ is a branch point of $\varphi_1$ and $z$ is not, then we can
choose coordinates as above so that locally $\Sigma_0 = \DD^2$ and
$Z = \DD^2 \times S^2$, with $\Pi(\zeta,w) = \zeta$ and a neighborhood of
$x$ in $C_1$ parametrized by a map in the form of \eqref{eqn:branchParam}
with $x = (0,0)$, while a neighborhood of $z$ in $C_1$ is a straightforward
section
$$
\DD^2 \hookrightarrow \DD^2 \times S^2 : \zeta \mapsto (\zeta,b)
$$
with $z = (0,b)$ for some constant $b \in \CC \setminus \{0\}$.
The neighborhood of $(z,x)$ in $C_1'$ is thus parametrized by the embedding
$$
h : \DD^2 \hookrightarrow C_1 \times \ZZ : \zeta \mapsto 
\left((\zeta^2,b),(\zeta^2,a \zeta)\right),
$$
which satisfies $\varphi_1'(h(z)) = (\zeta^2,b)$, showing that
$(z,x)$ is a simple branch point of~$\varphi_1'$.  The degree of
$\varphi_1'$ can be deduced by counting $\varphi_1^{-1}(z)$ for a generic
point $z \in C_1$: it is the set of all pairs
$(z,x)$ such that $x \in C_1$ as in the same fiber as~$z$ but is not equal
to it, so outside of the finitely many fibers that are not transverse 
to~$C_1$, the number of points in this set will be exactly $k_1-1$, and
they are all regular points by the discussion above.
\end{proof}

The next lemma follows from Lemma~\ref{lemma:notMinimal}:
\begin{lemma}
\label{lemma:stillNotMinimal}
For each singular fiber of $\Pi' : Z' \to C_1$, each of the two irreducible
components intersects $C'_1 \cup \ldots \cup C'_N$, and the intersections are transverse
and positive.\qed
\end{lemma}

\begin{lemma}
Each of the surfaces $C_1'', C_2',\ldots,C_N' \subset Z'$ has topological
type belonging to a finite list of possibilities that depend on the spinal
open book $\boldsymbol{\pi}$ and its chosen framing, but not on the filling~$W$.
\end{lemma}
\begin{proof}
For $i=2,\ldots,N$, $C'_i$ is a (possibly disconnected)
multisection of $\Pi' : Z' \to C_1$ with degree~$k_i$, and
$$
\psi_i := \Phi|_{C'_i} : C'_i \to C_i
$$
is a $k_1$-fold branched cover with exactly
$k_i$ simple branch points in every component of the branching locus
of~$\Phi$.
Then, using again the notation introduced following Corollary
\ref{cor:isABranchedCover}, where $b_i$ is the number of branch points of 
$\varphi_i = \Pi|_{C_i} \colon  C_i \to \Sigma_0$, this gives that 
$\psi_i \colon C'_i \to C_i$, is a $k_1$-fold branched cover with $k_ib_1$
simple branch points. 
Riemann-Hurwitz then implies
\[
    \chi(C'_i) = k_1 \chi( C_i) - k_i b_1.
\]
Recall from Corollary \ref{cor:isABranchedCover} and the subsequent discussion that Riemann-Hurwitz also
gives us
\[
    \chi(C_i) = k_i \chi(\Sigma_0) - b_i.
\]
It follows then 
\begin{align*}
    \chi(C'_i) &= k_1 \chi(C_i) - k_i b_1  \\
        &= k_1 \chi(C_i) - k_i ( k_1 \chi(\Sigma_0) - \chi(C_1) ) \\
        &= k_1 \chi(C_i) + k_i \chi(C_1) - k_1 k_i \chi(\Sigma_0).
\end{align*}
Recall that $\chi(\Sigma_0) \le 1$ since $\Sigma_0$ is connected (and has
boundary). This then gives a lower bound
\[
    k_1 \chi(C_i) + k_i \chi(C_1) - k_1 k_i \le \chi(C'_i),
\]
which only depends on the spinal open book decomposition. We also immediately
obtain the upper bound $\chi(C'_i) \le k_1 \chi(C_i)$, which also
depends only on the spinal open book decomposition. 

We now consider $C'_1 = \sigma_1 \cup C''_1$. 
Observe that the number of intersections of $\sigma_1$ with
$C''_1$ is precisely $b_1$, the number of branch points of $\varphi_1$. 
Applying the Riemann-Hurwitz formula and again using $\chi(\Sigma_0) \le 1$, 
we obtain the following bound that depends only on~$\boldsymbol{\pi}$:
\begin{equation}
\label{eqn:sigma1Int}
0 \le [\sigma_1] \cdot [C''_1] = -\chi(C_1) + k_1 \chi(\Sigma_0) \le
k_1 - \chi(C_1).
\end{equation}

The topologies of $\sigma_1$ and
$C''_1$ satisfy similar bounds in terms of~$\boldsymbol{\pi}$: $\sigma_1$
is diffeomorphic to $C_1$ since it is a section, 
and the topology of $C''_1$ can be bounded via the fact that
$\varphi_1' : C''_1 \to C_1$ is a $(k_1-1)$-fold branched cover
with $k_1 - 2$ simple branch points in each connected component of the
branching locus of~$\Phi$, i.e.~it has $(k_1-2)b_1$ branch points.
Again by Riemann-Hurwitz, 
$$-\chi(C''_1) + (k_1 - 1) \chi(C_1) = (k_1 - 2) b_1 = (k_1 -2) ( -\chi(C_1) + k_1 \chi(\Sigma_0) ).$$
Hence, 
$$
\chi(C''_1) = (2k_1 - 3) \chi(C_1) - k_1 (k_1 - 2) \chi(\Sigma_0).
$$
This is bounded below by $(2 k_1 - 3) \chi(C_1) - k_1 (k_1 - 2)$
since $\chi(\Sigma_0) \le 1$, and above by
$(2 k_1 - 3) \chi(C_1) - (k_1 - 2) \chi(C_1) = (k_1 - 1) \chi(C_1)$
due to \eqref{eqn:boundSigma0}.
\end{proof}

We now proceed by adapting the blow-down argument that worked in the case
$k_1 = 1$, but with $Z$ replaced by $Z'$ and the tautological
section $\sigma_1$ playing the role previously played by~$C_1$.
Let $\tau$ denote the framing of the link $\coprod_{i=1}^N \p C_i' \subset
\p Z'$ defined by lifting the framing of $\coprod_{i=1}^N \p C_i \subset \p Z$.
In light of the bounds we have already established on the topology of~$\Sigma_0$,
the next lemma implies both lower and upper bounds on the value of
$[\sigma_1] \cdot_\tau [\sigma_1]$.

\begin{lemma}
\label{lemma:sigma1Selfint}
$[\sigma_1] \cdot_\tau [\sigma_1] = \chi(C_1) - k_1 \chi(\Sigma_0)$.
\end{lemma}
\begin{proof}
We construct a perturbed section $\sigma_1^\epsilon \subset Z'$ as follows.
Each branch point of $\varphi_1 : C_1 \to \Sigma_0$ corresponds to a point
at which $C_1$ touches a fiber tangentially, and the Riemann-Hurwitz formula
dictates that the number of such points is $-\chi(C_1) + k_1 \chi(\Sigma_0)$.
Now consider a coordinate neighborhood of such a point as in 
Lemma~\ref{lemma:coordinates}: the local model near this intersection is
$$
C_1 = \left\{ (z^2,az)\ \in \DD^2 \times S^2\ \big|\ z \in \DD^2 \right\}
$$
for some $a \in \CC \setminus \{0\}$.  For $\epsilon > 0$ small, define the
perturbed surface
$$
C_1^\epsilon := \left\{ (z^2,az + \epsilon) \in \DD^2 \times S^2 \ \big|\ z \in \DD^2 \right\},
$$
which has a single transverse and positive intersection with $C_1$ in this
neighborhood.  
Now extend $C_1^\epsilon$ globally to a small perturbation of
$C_1$ that is pushed in the direction of the framing $\tau$ at the boundary,
and assume it is generic so that all other intersections between $C_1$ and
$C_1^\epsilon$ are transverse and occur in fibers that are transverse to~$C_1$.
Since $[C_1] \cdot_\tau [C_1] = 0$, the positive intersections we see in
the coordinate neighborhoods above must be canceled by further intersections,
the total signed count of which is therefore $\chi(C_1) - k_1 \chi(\Sigma_0)$.
But $C_1^\epsilon$ can also be lifted to a section $\sigma_1^\epsilon$
of $\Pi' : Z' \to C_1$ that will be disjoint from $\sigma_1$ near the
branching locus, while each of the other intersections of $C_1$ with
$C_1^\epsilon$ in $Z$ lifts to an intersection
of $\sigma_1$ with $\sigma_1^\epsilon$ in~$Z'$, giving
$$
[\sigma_1] \cdot_\tau [\sigma_1] = [\sigma_1] \cdot [\sigma_1^\epsilon] =
\chi(C_1) - k_1\chi(\Sigma_0).
$$
\end{proof}

The rest of the argument is completely analogous to the $k_1=1$ case, so
a sketch should now suffice.  Recall that the goal is to establish a
bound on the number $m \ge 0$ of singular fibers in $\Pi : Z \to \Sigma_0$.
Using the local model in Lemma~\ref{lemma:coordinates} to understand the branching
locus of $\Phi : Z' \to Z$, there is a well-defined almost complex
structure $J' := \Phi^*J$ on $Z'$ for which the multisections
$\sigma_1,C_1'',C_2',\ldots,C_N'$ are all $J'$-holomorphic.
Each of the singular fibers of $\Pi' : Z' \to C_1$
contains a unique exceptional sphere disjoint from~$\sigma_1$, 
so denote these by $E_1,\ldots,E_{k_1 m}$ and blow them down
to produce a smooth $S^2$-fibration
$$
\widecheck{\Pi}' : \widecheck{Z}' \to C_1
$$
with $\widecheck{J}'$-holomorphic fibers for some almost complex structure
such that the blowdown map $(Z',J') \to (\widecheck{Z}',\widecheck{J}')$
is pseudoholomorphic.  Composing the latter with our multisections in $Z'$ gives
rise to immersed $\widecheck{J}'$-holomorphic multisections
$$
\widecheck{\sigma}_1,\widecheck{C}''_1,\widecheck{C}'_2,\ldots,\widecheck{C}'_N 
\looparrowright \widecheck{Z}'.
$$
Trivializing the fibration $\widecheck{\Pi}' : \widecheck{Z}' \to C_1$
then identifies $\widecheck{Z}'$ with $C_1 \times S^2$ so that the framed
link $\coprod_{i=1}^N \p \widecheck{C}_i' \subset \p\widecheck{Z}'$ belongs
to one of a finite collection of isotopy classes of framed links in
$\p C_1 \times S^2$ that are independent of the choice of filling.
Using the established bounds on $[\widecheck{\sigma}_1] \cdot_\tau [\widecheck{\sigma}_1]
= [\sigma_1] \cdot_\tau [\sigma_1]$ and the fact that 
$[\widecheck{\sigma}_1] \cdot [F] = 1$ for the fiber class
$[F] \in H_2(C_1 \times S^2)$, this framed link determines 
$[\widecheck{\sigma}_1] \in H_2(C_1 \times S^2,\p\widecheck{\sigma}_1)$
up to a finite ambiguity.  Similarly, $\sigma_1$ has exactly
$-\chi(C_1) + k_1 \chi(\Sigma_0)$ positive and transverse intersections
with $C_1''$, which does not change after blowing down, nor does the fact
that $\sigma_1$ and $C_i'$ are disjoint for all $i=2,\ldots,N$, hence
these relations also determine the relative homology classes
$[\widecheck{C}_1'']$ and $[\widecheck{C}_i']$ for $i=2,\ldots,N$ up to
finite ambiguity.  These relative homology classes together with the
respective Euler characteristics determine the relative first Chern numbers
of the normal bundles of these surfaces, which are then used to bound the
numbers
$$
m_1 := \sum_{j=1}^{k_1 m} [C_1''] \cdot [E_j], \qquad\text{ and }\qquad
m_i := \sum_{j=1}^{k_1 m} [C_i'] \cdot [E_j] \text{ for $i=2,\ldots,N$}.
$$
By Lemma~\ref{lemma:stillNotMinimal}, we also have
$$
[C''_1] \cdot [E_j] + \sum_{i=2}^N [C'_i] \cdot [E_j] \ge 1
$$
for each $j=1,\ldots, k_1 m$, and thus conclude
$$
k_1 m \le \sum_{j=1}^{k_1 m} \left( [C''_1] \cdot [E_j] +
\sum_{i=2}^N [C'_i] \cdot [E_j] \right) 
= \sum_{i=1}^N m_i,
$$
giving an upper bound on $m$ that depends only on the framed
spinal open book~$\boldsymbol{\pi}$.
This concludes the proof in the general case.

\end{document}